\documentclass[11pt]{article}

\usepackage{epsfig}
\usepackage{amssymb, latexsym}
\usepackage{amscd}
\usepackage[all]{xy}
\usepackage{epsf}
\usepackage[mathscr]{eucal}
\usepackage{tikz}
\usepackage{tikz-cd}
\usepackage{verbatim}
\usetikzlibrary{matrix,arrows}
\usepackage{mathtools}
\usepackage{xcolor}
\usepackage{stmaryrd}
\usetikzlibrary{positioning,fit,calc}

\DeclareFontFamily{U}{solomos}{}
\DeclareErrorFont{U}{solomos}{m}{n}{10}
\DeclareFontShape{U}{solomos}{m}{n}{
  <-> s*[1.1]  gsolomos8r
}{}

\usepackage{amsmath,amsfonts,amscd,amssymb,amsthm}

\usepackage[titletoc]{appendix}
\usepackage[sc]{titlesec}

\long\def\comment#1\endcomment{}

\theoremstyle{plain}

\makeatletter
\newcommand*{\doublerightarrow}[2]{\mathrel{
  \settowidth{\@tempdima}{$\scriptstyle#1$}
  \settowidth{\@tempdimb}{$\scriptstyle#2$}
  \ifdim\@tempdimb>\@tempdima \@tempdima=\@tempdimb\fi
  \mathop{\vcenter{
    \offinterlineskip\ialign{\hbox to\dimexpr\@tempdima+1em{##}\cr
    \rightarrowfill\cr\noalign{\kern.5ex}
    \rightarrowfill\cr}}}\limits^{\!#1}_{\!#2}}}
\makeatother

\newtheorem{theorem}{\sc Theorem}[section]
\newtheorem{lemma}[theorem]{\sc Lemma}

\newtheorem{prop}[theorem]{\sc Proposition}

\theoremstyle{plain}

\newtheorem{defn}[theorem]{\sc Definition}

\theoremstyle{exercise}
\newtheorem{remark}[theorem]{\sc Remark}

\newtheorem{example}[theorem]{\sc Example}

\makeatletter \@addtoreset{equation}{section} \makeatother

\def\eqref#1{\thetag{\ref{#1}}}

\let\latexref=\ref
\def\ref#1{{\normalfont{\latexref{#1}}}}

\newcommand{\ldot}{{\:\raisebox{2,3pt}{\text{\circle*{1.5}}}}}
\newcommand{\udot}{{\:\raisebox{3pt}{\text{\circle*{1.5}}}}}
\def\dlim_#1{{\displaystyle\lim_{#1}}^\hdot}

\newcommand{\id}{\operatorname{\rm id}}

\newcommand{\Mor}{\mathrm{Mor}}
\newcommand{\Ob}{\mathrm{Ob}}

\newcommand{\opp}{\mathrm{opp}}

\newcommand{\Hom}{\mathrm{Hom}}

\newcommand{\Hoch}{\mathrm{Hoch}}

\newcommand{\dg}{\mathrm{dg}}

\newcommand{\Cone}{\mathrm{Cone}}

\newcommand{\Sets}{\mathscr{S}ets}

\newcommand{\Cat}{{\mathscr{C}at}}

\renewcommand{\k}{\Bbbk}

\newcommand{\pprime}{{\prime\prime}}

\newcommand{\tot}{\mathrm{tot}}

\newcommand{\coh}{\mathrm{coh}}
\newcommand{\Coh}{{\mathscr{C}{oh}}}

\newcommand{\sotimes}{\overset{\sim}{\otimes}}

\newcommand{\Ord}{\mathrm{Ord}}

\newcommand{\Tot}{\mathrm{Tot}}

\newcommand{\Sym}{\mathrm{Sym}}

\renewcommand{\min}{\mathrm{min}}
\renewcommand{\max}{\mathrm{max}}

\newcommand{\Op}{\mathrm{Op}}

\newcommand{\Nor}{\mathrm{Nor}}
\newcommand{\Tree}{\mathbf{Tree}}
\newcommand{\Out}{\mathrm{Out}}
\newcommand{\Symm}{\mathrm{Symm}}
\newcommand{\Des}{\mathrm{Des}}
\newcommand{\Trees}{\mathbf{Tree}}
\newcommand{\dom}{\mathrm{dom}}
\newcommand{\Glob}{\mathbf{Glob}}

\newcommand{\sevafigc}[4]{\begin{figure}[h]\centerline{
 \epsfig{file=#1,width=#2,angle=#3}}
\bigskip\caption{#4}\end{figure}}

\title{\sc{On comparison of the Tamarkin and the twisted tensor product 2-operads}}

\author{\sc{Boris Shoikhet}}
\date{}

\begin{document}\maketitle
{\footnotesize
\begin{center}{\parbox{4,5in}{{\sc Abstract.}
There are known two different constructions of contractible dg 2-operads, providing a weak 2-category structure on the following dg 2-quiver of small dg 2-categories. Its vertices are small dg 2-categories over a given field, arrows are dg functors, and the 2-arrows $F\Rightarrow G$ are defined as the Hochschild cochains of $C$ with coefficients in $C$-bimodule $D(F(-),G(=))$, where $F,G\colon C\to D$ are dg functors, $C,D$ small dg categories. It is known that such definition is correct homotopically, but, on the other hand, the corresponding dg 2-quiver fails to be a strict 2-category.
 The question ``What do dg categories form'' is the question of finding a weak 2-category structure on it, in an appropriate sense. 
One way of phrasing it out is to make it an algebra over a contractible 2-operad, in the sense of M.Batanin [Ba1,2] (in turn, there are many compositions of 2-arrows for a given diagram, but their totality forms a contractible complex) . 

In [T], D.Tamarkin proposed a contractible $\Delta$-colored 2-operad in Sets, whose dg condensation solves the problem. 
In our recent paper [Sh2], we constructed contractible dg 2-operad, called the twisted tensor product operad, acting on the same 2-quiver  (the construction uses the twisted tensor product of small dg categories [Sh1]).  

The goal of this paper is twofold. At first,
we establish links  between the twisted tensor product dg 2-operad and the dg condensation of the Tamarkin simplicial 2-operad.
At second, we develop computational tools, generalising some constructions of [BBM] from the case of symmetric $\Delta$-colored lattice path operad [BB] to the case of $\Delta$-colored  Tamarkin 2-operad.

We construct a 2-operadic analogue of the dg operad $\mathcal{B}r$ from [BBM] (called $\mathbf{Breq}$), and prove that it is quasi-isomorphic both to its normalised quotient $\Nor(\mathbf{Breq})$ and to the dg condensation of the Tamarkin simplicial 2-operad. We prove the contractibility 
of the 2-operad $\Nor(\mathbf{Breq})$ directly. As well, we establish an isomorphism of dg 2-operads between $\Nor(\mathbf{Breq})$ and the twisted tensor product 2-operad (the contractibility of the latter operad is proven in [Sh2] by homotopical methods). As a byproduct, we provide two new proofs of contractibility of the dg condensation of Tamarkin 2-operad.

}}
\end{center}
}

\section*{\sc Introduction}
Throughout the paper, $\k$ is a field of characteristic 0. 

In their pioneering paper [GJ], Getzler and Jones found, among the other things, an operad called $B_\infty$, whose quotient, called the brace operad, acts on the cohomological Hochschild complex of any dg algebra (and, more generally, of any dg category). \footnote{The Getzler-Jones brace operad is $\Nor(\mathcal{B}r_2)$, in notations of [BBM].} 
The cohomology of the Getzler-Jones brace operad is isomorphic, an a graded operad, to $e_2=H_\udot(E_2,\k)$, and the operad itself is quasi-isomorphic to $e_2$ (although the latter statement is much harder and is one of the versions of the Deligne conjecture). 

Inspired by McClure-Smith approach to the Deligne conjecture [MS1,2], Batanin and Berger [BB] introduced a simplicial operad, called the lattice path operad, equipped with an ascending operad filtration, with filtration components $\mathcal{L}_n$. 
It is proven [BB] that the topological condensation of $\mathcal{L}_n$ is a $E_n$-operad, and the dg condensation of $\mathcal{L}_n$ is a $C_\ldot(E_n;\k)$-operad. 
Thus the lattice path operad with its filtation provides an explicit realisation of the operad $E_\infty$ with its ascending filtration by operads $E_n$ [B]. In [BBM], a dg brace operad in any dimension $n$ was introduced, as
$$
\mathcal{B}r_n(\ell)^{-p}=\bigoplus_{k_1+\dots+k_\ell=p}\mathcal{L}_n(k_1,\dots,k_\ell; 0)
$$
with the differential coming from the elementary face operations. 

A priori $\mathcal{B}r_n$ is not an operad, and is made an operad via 
the {\it whiskering map} [BBM, Sect. 3]
$$
w\colon \mathcal{B}r_n(\ell)\to |\mathcal{L}_n|(\ell)
$$
where $|-|$ stands for the non-normalised dg condensation,
$$
|X|^{-p}=\prod_{s\ge 0}\bigoplus_{k_1+\dots+k_\ell-s=p}X(k_1,\dots,k_\ell; s)
$$

In this paper, we consider a similar to [BBM] paradigm for the case of $n$-operads (restricting ourselves by the case $n= 2$), when the Tamarkin simplicial 2-operad $\mathbf{seq}$ (see Section 1) replaces the second stage $\mathcal{L}_2$ of the lattice path operad. 
We construct the corresponding brace 2-operad (called $\mathbf{Breq}$),\footnote{The name originates from author's attempt to hybridize {\it brace} and {\it seq}} and a whiskering map $\mathbf{Breq}\to |\mathbf{seq}|$ (which makes $\mathbf{Breq}$ a dg 2-operad).
We also consider the normalised Moore complex components $\Nor(\mathbf{Breq})$, which give rise to another dg 2-operad, and prove that the two maps
$$
|\mathbf{seq}|\overset{w}{\leftarrow}\mathbf{Breq}\to\Nor(\mathbf{Breq})
$$
are quasi-isomorphisms of dg 2-operad (Proposition \ref{whiskprop}).

We show that the operad $\Nor(\mathbf{Breq})$ has very manageable components, whose contractibility to $\k[0]$ can be proved 
by a direct computation (Theorem \ref{theoremb3}). As a consequence, it follows that (the dg condensation of) Tamarkin 2-operad $|\mathbf{seq}|$ is contractible. 

In [Sh2], we gave a new dg 2-operad acting on the 2-quiver $\Cat_\dg^\coh(\k)$, which looked much smaller than (the dg condensationof ) the Tamarkin 2-operad $\mathbf{seq}$. It was defined via the twisting tensor product [Sh1], and is denoted by $\mathcal{O}$. (We recall the construction in Section 5). 
We identify the dg 2-operad $\mathcal{O}$ with the dg 2-operad $\Nor(\mathbf{Breq})$, so that the two 2-operads become isomorphic. 
Then Proposition 5.7 of [Sh2] (which says that the twisted tensor product operad is contractible) and Proposition \ref{ocontr} together provide another proof of contractibility for the dg condensation of the Tamarkin 2-operad $\mathbf{seq}$. 

\subsection*{Acknowledgements}
 The author is thankful to the anonymous referee of [Sh2] who asked him on possible links between the twisted product 2-operad $\mathcal{O}$ and $\mathbf{seq}$, and asked whether the whiskering map from [BBM] can be defined in the 2-operadic context. This paper grew up from an attempt to address this question.

\vspace{1mm}

The work was financed 
by the Support grant
for International Mathematical Centres Creation and Development, by the Ministry of Higher Education and Science
of Russian Federation and PDMI RAS agreement № 075-15-2022-289 issued on April 6, 2022.


\section{\sc Tamarkin's 2-operad}\label{tamopapp1}
\subsection{\sc Reminder on 2-operads}
Here we fix our notations in 2-operads, the reader is referred to [Ba1,2], [T], [BM1,2] for more detail. 
\subsubsection{}
An $n$-operad for $n=1$ is just a non-symmetric (non-$\Sigma$) operad. Recall that a non-$\Sigma$ operad $\mathcal{O}$ in a monoidal category $M$ is given by its arity components $\mathcal{O}(0), \mathcal{O}(1),\mathcal{O}(2),\dots$ which are elements in $M$, along with  {\it operadic composition}
\begin{equation}\label{eqop1}
m\colon \mathcal{O}(k)\otimes\mathcal{O}(\ell_1)\otimes\dots\otimes \mathcal{O}(\ell_k)\to \mathcal{O}(\ell_1+\dots+\ell_k)
\end{equation}
and the unit $\id_1\in \mathcal{O}(1)$. 

The operadic composition can be viewed as been associated with a map of ordinals $p\colon [\ell_1+\dots+\ell_k-1]\to [k-1]$ defined as $p^{-1}(i)=\ell_{i+1}$. Namely, to any map of ordinals $\phi\colon [m-1]\to [k-1]$ one associates the set of {\it fibers} $\{\phi^{-1}(i)\}_{i\in [k-1]}$, and the operadic composition is a map
\begin{equation}\label{eqop2}
\mathcal{O}([k-1])\otimes \bigotimes_{i\in [k-1]}\mathcal{O}(\phi^{-1}(i))\to \mathcal{O}([m-1])
\end{equation}
The fibers $\{\phi^{-1}(i)\}_{i\in [k-1]}$ are thought of as ordinals again. If the map $\phi$ is surjective, all fibers are non-empty, $\ell_i\ge 1$. In general, the surjectivity of $\phi$ is not required, and $\mathcal{O}(\emptyset)$ is considered as $\mathcal{O}(0)$, in this way one can talk about {\it algebra units} (an element $\id_0\in\mathcal{O}(0)$ acts on any $\mathcal{O}$-algebra $V$ as the substitution of the unit $1_V\in V$).

The operadic composition and operad identity are subject to the following axioms:
\begin{itemize}
\item[(1)] the associativity of the operadic composition,
\item[(2)] for any $k\ge 1$, $m(\mathcal{O}(k), \id_1,\dots,\id_1)=\id_{\mathcal{O}(k)}$,
\item[(3)] for any $k\ge 0$, $m(\id_1\otimes \mathcal{O}(k))=\id_{\mathcal{O}(k)}$
\end{itemize}

Roughly speaking, the idea of higher operads of Batanin was to consider, for each $n\ge 1$, a category $\Tree_n$, such that $\Tree_1=\Delta$, in which fibers are defined, and mimick the definition \eqref{eqop2} (see below). An $n$-operad $\mathcal{O}$ is defined as a sequence of objects $\{\mathcal{O}(T)\}_{T\in\Tree_n}$, such that \eqref{eqop2} holds, and with suitably defined unit and compatibilities. In particular, the composition \eqref{eqop2} should be associative.

For the case $n=1$, an action of a non-$\Sigma$ operad $\mathcal{O}$ on $V\in M$, where $M$ is a monoidal category with weak equivalences, such that, for any $k\ge 0$, $\mathcal{O}(k)$ is weakly equivalent to the final object $1\in M$, is considered as an $A_\infty$ monoid structure on $V$ (that is, as a structure of a weak associative algebra). 

The idea is that a weak $n$-algebra can be defined as an algebra over $n$-operad $\mathcal{O}$, such that for each $T\in \Tree_n$, $\mathcal{O}(T)$ is weakly equivalent to the final object $1\in M$. (More precisely, the collection of weak equivalences $\{\mathcal{O}(T)\to 1\}_{T\in\Tree_n}$ should be compatible with the operadic composition). In particular, if $\mathcal{O}(T)=1$ for any $T$, we recover strict $n$-algebras. See Theorem \ref{theoremm} (due to Batanin) below for the precise statement. 

\subsubsection{}
Denote by $\{k\}$ the underlying finite set of the ordinal $[k]$. 
The category $\Tree_n$ is defined as follows. Its object $T$ is an $n$-string of surjective maps in $\Delta$:
$$
T=[k_n-1]\xrightarrow{\rho_{n-1}} [k_{n-1}-1]\xrightarrow{\rho_{n-2}}\dots\xrightarrow{\rho_0}[0]
$$
Such $T$ is visualised as a $n$-level tree. An $n$-tree is called {\it pruned} if all $\rho_i$ are surjective. For a pruned $n$-tree, all its leaves are at the highest level $n$. The finite set of leaves of an $n$-tree $T$ is denoted by $|T|$.
The maps $\rho_i$ are referred to as the {\it structure maps} of an $n$-tree.

A morphism $F\colon T\to S$, where 
$$
S=[\ell_n-1]\xrightarrow{\xi_{n-1}} [\ell_{n-1}-1]\xrightarrow{\xi_{n-2}}\dots\xrightarrow{\xi_0}[0]
$$
is defined as a sequence of maps $f_i\colon \{k_i-1\}\to \{\ell_i-1\}$, $i=0,1,\dots,n$ (not monotonous, in general), which commute with the structure maps, and such that  
for each $0\le i\le n$ and each $j\in [k_{i-1}-1]$ the restriction of $f_{i}$ on $\rho_{i-1}^{-1}(j)$ is monotonous. That is, $f_i$ has to be monotonous when restricted to the fibers of the structure map $\rho_{i-1}$. 
It is clear that a map of $n$-trees is uniquely defined by the map $f_n$. Conversely, any map $f_n$ which is a map of {\it $n$-ordered sets}, associated with $n$-trees $S$ and $T$, defines a map of $n$-trees (see [Ba3, Lemma 2.3]). 

The fiber $F^{-1}(a)$ for a morphism $F\colon T\to S$, $a\in |S|$, is defined as the set-theoretical preimage of the linear subtree $\Out(a)$ of $S$ spanned by $a$. This linear subtree $\Out(a)$ is defined as follows. Let $a\in [\ell_i]$, then $\Out(a)$ has no vertices at levels $>i$, and the only vertex of $\Out(a)$ at level $j\le i$ is defined as $\rho_j\dots\rho_{i-2}\rho_{i-1}(a)$. Note that a fiber of a map of pruned $n$-trees is not necessarily a pruned $n$-tree, even if all components $\{f_i\}$ of the map $F$ are surjective, see Remark \ref{remprune}.

\begin{example}{\rm Consider the case $n=2$, $T=[3]\xrightarrow{\rho_1} [1] \to [0]$, $S=[1]\to [0]\to [0]$. Denote by $0<1$ the leaves of $S$. Define maps $F_1, F_2\colon T\to S$ as follows: $F_1(0)=F_1(1)=0, F_1(2)=F_1(3)=1$, and $F_2(0)=F_2(2)=0, F_2(1)=F_2(3)=1$. Both $F_1,F_2$ are maps of level trees. Note that the map $F_2$ is not defined via an ordinal map $f\colon [3]\to [1]$, as $f(1)>f(2)$. At the same time, the restriction of $f$ on each fiber $f^{-1}(i)$, $i=0,1$, is a map of ordinals. }
\end{example}

\begin{remark}\label{remprune}{\rm
(1) Let $T,S$ be pruned $n$-trees, $\sigma\colon T\to S$ a map of $n$-trees. Note that the fibers $F^{-1}(a)$, $a\in |S|$, needn't be pruned $n$-trees, even if the components $f_i$ are surjective. An example in shown in Figure \ref{fig1}. 
For a possibly non-pruned $n$-tree $T$, denote by $P(T)$ the maximal pruned $n$-subtree of $T$. By definition, it is the pruned $n$-tree generated by all level $n$ leaves of $T$, by ignoring the leaves at levels $<n$ as well as their descendants. We call $P(T)$ the {\it prunisation} of $T$.

\sevafigc{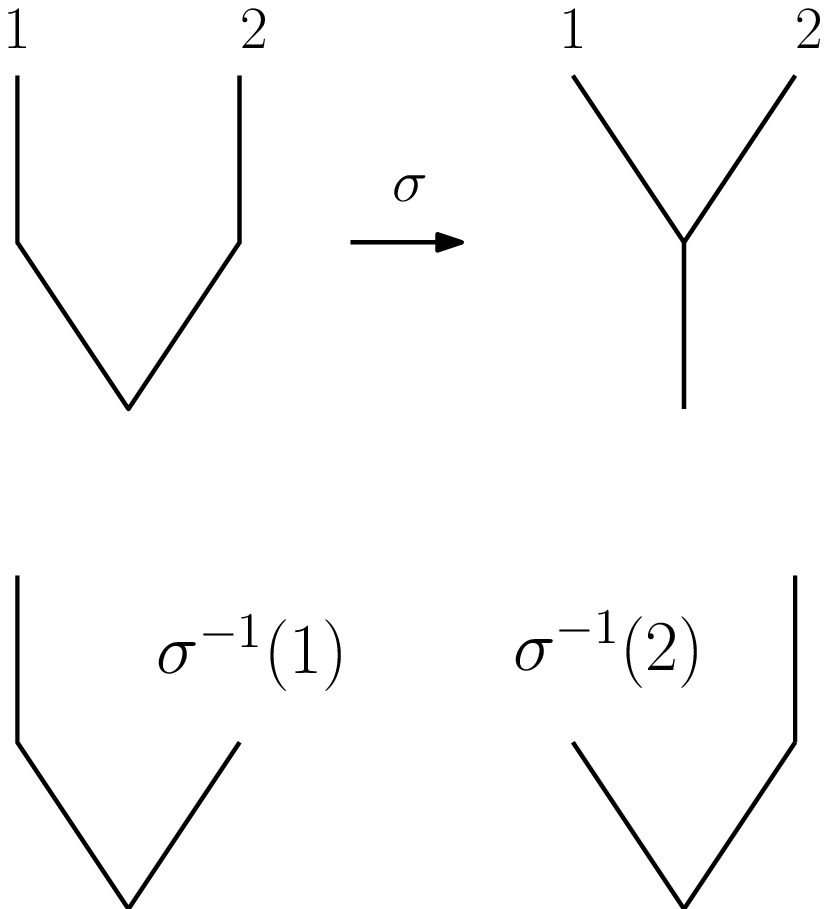}{45mm}{0}{\label{fig1}}

(2) As categories, the pruned $n$-trees and $n$-ordinals $\mathbf{Ord}_n$ [Ba2] Def. 2.2 are the same. On the other hand, as {\it operadic categories}
$\Tree_n$ and $\Ord_n$ are different: for a morphism $\sigma \colon T\to S$ of pruned $n$-trees, a fiber $\sigma_{\Ord_n}^{-1}(i)$ in $\Ord_n$ is defined as the prunisation $P(\sigma_{\Trees_n}^{-1}(i))$. 
}
\end{remark}

The linear pruned $n$-level tree $U_n$ (having a signle element at each level) is the final object in the category of pruned $n$-trees.

\subsubsection{}
We recall here the definition of a {\it pruned} {\it reduced} $n$-operad. In terminology of [Ba3], we consider here only {\it $(n-1)$-terminal} $n$-operads. The $(n-1)$-terminality makes us possible to restrict with $n$-operads taking values in a symmetric monoidal globular category $\Sigma^nV$, where $V$ is a closed symmetric monoidal category, see [Ba2], Sect. 5. By a slight abuse of terminology, we say that an operad takes values in the closed symmetric monoidal category $V$ (not indicating $\Sigma^nV$).

\begin{defn}{\rm

A pruned reduced $(n-1)$-terminal $n$-operad $\mathcal{O}$ in a symmetric monoidal category $V$ is given by an assignment $T\rightsquigarrow \mathcal{O}(T)\in V$, for a pruned $n$-tree $T$, so that for any {\it surjective} map $\sigma\colon T\to S$ of pruned $n$-trees, one is given the composition
\begin{equation}
m_\sigma\colon \mathcal{O}(S)\otimes \mathcal{O}(P(\sigma^{-1}(1)))\otimes\dots\otimes \mathcal{O}(P(\sigma^{-1}(k)))\to \mathcal{O}(T)
\end{equation}
where $k=|S|$ is the number of leaves of $S$, and $P(-)$ is the prunisation (which cuts all non-pruned branches, see Remark \ref{remprune}). It is subject to the following conditions (in which we assume that $V=C^\udot(\k)$ is the category of complexes of $\k$-vector spaces):
\begin{itemize}
\item[(i)] $\mathcal{O}(U_n)=\k$, and $1\in \k$ is the operadic unit,
\item[(ii)] the associativity for the composition of two surjective morphisms
$T\xrightarrow{\sigma}S\xrightarrow{\rho} Q$ of pruned $n$-trees, see [Ba2] Def. 5.1,
\item[(iii)] the two unit axioms, see [Ba2], Def. 5.1.
\end{itemize}

The category of pruned reduced $(n-1)$-terminal $n$-operads in a symmetric monoidal category $V$ is denoted by $\Op_n(V)$, or simply by $\Op_n$. 
}
\end{defn}

\begin{remark}{\rm
The idea behind the definition of pruned reduced operad is that algebras over such operads should be {\it strictly unital}. The fact that we can cut off all not pruned branches means that these redundant pieces act by (whiskering with) the identity morphism. When we deal with algebras with weak units, we have to consider more general $n$-operads. 
}
\end{remark}

\subsubsection{\sc }

Tamarkin [T] uses the Joyal dual description of the category $\Trees_n$ via Joyal $n$-disks. This description is more ``globular'', and it fits better for the questions of globular nature, such as for describing a weak 2-category structure on all dg categories, via a suitable 2-operad. Let us recall it, for simplicity restricting ourselves with the case $n=2$. 

A globular 2-diagram $D$ is given by  sets, $D_0, D_1, D_2$, and the following maps
\begin{equation}\label{globpres}
D_2\doublerightarrow{s_1}{t_1}D_1\doublerightarrow{s_0}{t_0}D_0
\end{equation}
such that $s_0s_1=s_0t_1$, $t_0s_1=t_0s_1$. A morphism $D\to D^\prime$ of two 2-globular diagrams is defined as collection of maps $\{D_i\to D_i^\prime\}_{i=0,1,2}$, which commute with all $s_j$ and $t_j$. 

A globular 2-diagram $D$ is considered as a quiver for generating a strict 2-category, which we denote by $\omega_2(D)$. The functor $D\rightsquigarrow\omega_2(D)$ is the left adjoint to the forgetful functor from the category of strict 2-categories to globular diagrams. 

\sevafigc{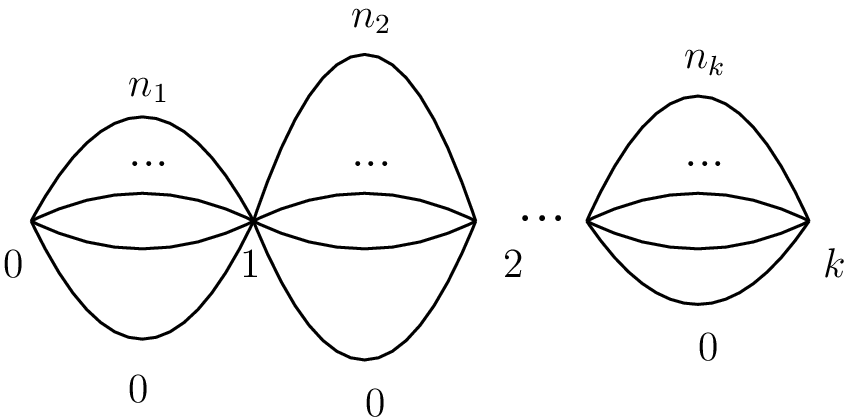}{65mm}{0}{The 2-globular diagram $D=(n_1,\dots,n_k)$\label{fig2} }

Assume we are given a globular diagram as at the Figure \ref{fig2}, then one associates to it the 2-tree 
$$T(D)= [n_1+\dots+n_k-1]\xrightarrow{\rho} [k-1]\to [0]$$
where $\sharp\{\rho^{-1}(i)\}=n_i$ and the map $\rho$ is monotonous surjective. The $n_1, n_2,\dots,n_k$ leaves at the corresponding branches of the tree $T(D)$ are thought of as the elementary generating 2-morphisms (thus, the vertical intervals) in $D$.
We denote such globular diagram by $D=(n_1,\dots,n_k)$. As we model the category of pruned 2-trees, we impose the conditions
$$
k\ge 1, n_i\ge 1 \text{  for } 1\le i \le k 
$$

We define a category whose objects are 2-globular diagrams. We define a map $f\colon D\to D^\prime$ as a strict 2-functor between the strict 2-categories $f\colon \omega_2(D)\to \omega_2(D^\prime)$ which is {\it dominant} (see Remark \ref{remberger} below) in the following sense. The objects of $\omega_2(D)$ are linearly ordered, and for any two objects $i\le j$, the set of 1-morphisms $\omega_2(D)(i,j)$ is partially ordered, with the minimal and the maximal elements. The dominance of $f$ means that $\omega_2)(f)$ preserves the minimal and the maximal vertices, and for any $i\le j\in \omega_2(D)$, $\omega_2(f)$ maps the minimal element of $\omega_2(D)(i,j)$ to the minimal element in $\omega_2(D^\prime)(f(i),f(j))$, and similarly for the maximal elements. We denote the category whose objects are 2-globular diagrams, and morphism are strict dominant 2-functors as above, by $\Glob_2^\dom$. For $D\in \Glob_2^\dom$, $D=(n_1,\dots,n_k)$, we assume that $k\ge 1$, $n_1,\dots,n_k\ge 1$. 

One has:
\begin{prop}\label{propduality}
The category $\Glob_2^\dom$ is anti-equivalent to the category $\Tree_2$ of pruned 2-trees.
\end{prop}
See [B], Prop. 2.2 (along with Remark \ref{remberger}, and [T], 4.1.9-10 for a proof. 

\begin{remark}\label{remberger}{\rm
One can define another category whose objects are 2-globular diagrams, by $\Hom(D, D^\prime):=\Hom_{2{-}\Cat}(\omega_2(D),\omega_2(D^\prime))$. In this way, we get the category $\Theta_2$, see [B]. It is proven [B] Prop. 2.2 that this category is dual to the category of Joyal 2-disks [J]. One can show that this duality restricted to the dominant morphisms of 2-categories provides an (anti-) equivalence with the subcategory of the category Joyal 2-disks, whose morphisms have the following property: the only pre-image of any boundary point is a boundary point. The latter category is clearly identified with the category of 2-level trees. 
}
\end{remark}

Of course, the duality of Proposition \ref{propduality} is the $n=2$ analogue of the classical Joyal duality between the category $\Delta_+$ (the usual category $\Delta$ augmented with an initial object [-1]) and the category $\Delta_{fi}$ of {\it finite intervals}, see [J].

\subsubsection{\sc Batanin Theorem}
Denote the category of symmetric operads (in a given symmetric monoidal category) by $\Op_\Sigma$.

Batanin [Ba2], Sect. 6 and 8, constructs a pair of functors relating symmetric operads and $n$-operads:

$$
\Symm\colon \Op_n^{n-1}\rightleftarrows   \Op_\Sigma \colon \Des
$$
The right adjoint functor of desymmetrisation $\Des$ associates to each pruned $n$-tree $T$ its set of leaves $|T|$ (which are all at the level $n$):
$$
\Des(\mathcal{O})(T)=\mathcal{O}(|T|)
$$
and for a map $\sigma\colon T\to S$ of $n$-trees, the $n$-operadic composition associated with $\sigma$ is defined as the corresponding composition for $|\sigma|=|\sigma_n|\colon |T|\to |S|$, twisted by the shuffle permutation $\pi(\sigma_n)$ of the map $|\sigma_n|\colon |T|\to |S|$ defined by the condition that the composition of $\pi(\sigma)$ followed by an order preserving map of finite sets is $\sigma_n$ (see [Ba2], Sect. 6). 

The symmetrisation functor is defined as the left adjoint to $\Des$, its existence is established in [Ba2], Sect. 8.

The main theorem on $n$-operads was proven in [Ba3] Th.8.6 for topological spaces and in [Ba3] Th.8.7 for complexes of vector spaces. We provide below the statement for $C^\udot(\k)$, as the one we use here.
Denote by $\underline{\k}$ the constant $n$-operad, $\underline{\k}(T)=\k$, with evident operadic compositions. We say that an $n$-operad in $C^\udot(\k)$ is {\it augmented} by $\underline{\k}$ if there is a map of $n$-operads $p\colon \mathcal{O}\to\underline{\k}$, called the augmentation map. 

\begin{theorem}\label{theoremm}{\rm [Batanin]}
Let $\mathcal{O}$ be reduced pruned $(n-1)$-terminal $n$ operad in the symmetric monoidal category $C^\udot(\k)$. 
Assume $\mathcal{O}$ is augmented to the constant $n$-operad $\underline{\k}$, and that for any arity $T$ the augmentation map $p(T)\colon \mathcal{O}(T)\to \k$ is a quasi-isomorphis of complexes. 
Then there is a morphism of $\Sigma$-operads $C_\udot(E_n;\k)\to\Sym(\mathcal{O})$, thus making any $\mathcal{O}$-algebra a $C_\udot(E_n;\k)$-algebra.
\end{theorem}

\begin{remark}{\rm
There are closed model structures on the categories of $\Sigma$-operads and $n$-operads, constructed in [BB2]. Within these model structures, $(\Symm, \Des)$ is a Quillen pair, with $\Symm$ the left adjoint. The stronger version of this theorem [Ba3] actually says that the symmetrisation of a {\it cofibrant} contractible pruned, reduced, $(n-1)$-terminal is weakly equivalent to the symmetric operad $C^\udot(E_n;\k)$. 
}
\end{remark}

An advantage of the approach of Theorem \ref{theoremm} to $n$-algebras via contractible $n$-operads is that the latter is much simpler and more ``linear'' object than the symmetric operads $E_n$ and $e_n$. At the same time, it links higher category theory and $E_n$-algebras in a very explicit way.

\subsection{\sc Tamarkin's 2-operad}
\subsubsection{\sc Tamarkin's notations on 2-operads}\label{tamnot}
In this paper, we adopt the ``globular'' notations for $n$-operads, used in [T]. Thus, the components of an $n$-operad $\mathcal{O}$ are $\mathcal{O}(D)$ where $D$ is a globular diagram as above. We use the following notations:

We use notations $U,V,...$ for globular diagrams, $[U]$ for $\omega_2(U)$. Tamarkin calls a 2-globular diagram {\it a 2-disk} (this terminology is a bit confusing, because of Joyal $n$-disks). A {\it minimal ball} in $U$ is an elementary 2-morphism of $\omega_2(U)=[U]$, that is, an element of $U_2$ in the presentation \eqref{globpres}. Minimal balls are denoted by $\mu,\nu,\dots$. The set of all minimal balls in a 2-globular diagram $U$ is denoted by $\mathcal{F}(U)$. 

Denote by $\mathcal{C}(U)$ the set of {\it intervals} of $U$, that is, the set $U_1$ in the presentation \eqref{globpres}. (It is similar to the ordered set of intervals in an ordinal; for the ordinal $[k]=\{0<1<2<\dots<k\}$, the cardinality of the set of intervals is $k$, and the intervals are $(01), (12),\dots, (k-1,k)$. We sometimes write $\overrightarrow{i,i+1}$ for the interval $(i,i+1)$.

For a 2-globular set $U$, there is a map 
$$
\pi(U)\colon \mathcal{F}(U)\to\mathcal{C}(U)
$$
defined as $\pi_U(\nu)=(i,i+1)$ if $s_1(\nu)=i, t_1(\nu)=i+1$. 

\subsubsection{}
Tamarkin's 2-operad $\mathbf{seq}$ is a contractible 2-operad in sets whose colors are $[0],[1],[2],...$, and the category of unary operations is $\Delta$. 

Let $U$ be a 2-disk. Recall our notations $\mathcal{F}(U)$ for its minimal balls (which are the generators in the strict 2-category $[U]$), $\mathcal{C}(U)$ for its intervals, and $\pi_U\colon \mathcal{F}(U)\to\mathcal{C}(U)$ for the map defined as $\pi_U(\nu)=\overrightarrow{\ell}$ if $\nu\in U(\ell-1,\ell)$, see Section \ref{tamnot}. 

As a colored 2-operad in sets, $\mathbf{seq}$ is given by a collection of sets $\mathbf{seq}(U)(\{I_\nu\}_{\nu\in\mathcal{F}(U)}; J)$
where $\{I_\nu\}$ are the colors (1-ordinals) associated with all minimal balls of $U$, and $J$ is the output color. 

This 2-operad is defined in [T, Sect. 6.1], we slightly rephrase the definition.

The set $\mathbf{seq}(U)(\{I_\nu\}; J)$ is given by the following data:
\begin{itemize}
\item[A)] a total order on the disjoint union $\mathcal{I}_U:=\sqcup_{\nu\in\mathcal{F}(U)}I_\nu$ which is subject to the conditions (1)-(3) below;  we use notation $I_\tot$ for the 1-ordinal on $\mathcal{I}_U$.
\item[B)] a map $W\colon I_\tot\to J$ of 1-ordinals,
\end{itemize}
where the conditions (1)-(3) read:
\begin{itemize}
\item[(1)] the total order $I_\tot$ on $\mathcal{I}_U$ agrees with the order on each $I_\nu$,
\item[(2)] if $\pi_U(\nu_1)=\pi_U(\nu_2)$, and $\nu_1<\nu_2$, then $I_{\nu_1}<I_{\nu_2}$,
\item[(3)] if $\pi_U(\nu_1)<\pi_U(\nu_2)$, then the following 3 cases are possible:
\begin{itemize}
\item[(3a)] $I_{\nu_1}<I_{\nu_2}$,
\item[(3b)] $I_{\nu_1}>I_{\nu_2}$,
\item[(3c)] (aka ``the brace'') assume $I_{\nu_2}=[n]$, then there is $0\le a<n$ such that $[0,a]<I_{\nu_1}< [a+1,n]$
\end{itemize}
\end{itemize}
This definition indeed gives rise to a $\Delta$-colored 2-operad in sets, as we recall below.

\begin{remark}{\rm
The set $\mathbf{seq}(U)(\{I_\nu\};J)$ behaves (for a fixed $U$) contravariantly in each $I_\nu$ and covariantly in $J$. The latter is clear; to explain the former, assume we have maps $\phi_\nu\colon I_\nu^\prime\to I_\nu$ in $\Delta$. Then the total order $I_\tot$ on $\sqcup_\nu I_\nu$ can be pulled back to a total order $I_\tot^\prime$ on $\sqcup_\nu I_\nu^\prime$, and this pulled-back total order on $\sqcup I_\nu^\prime$ satisfies (1)-(3) above if the original one does. 
Moreover, one gets a map of the total order 1-ordinals $\phi_{\nu *}\colon I_\tot^\prime\to I_\tot$, so that one defines a map $W^\prime\colon I_\tot^{\prime}\to J$ as $W\colon=W\circ \phi_{\nu *}$.  In this way one gets a functor $\mathbf{seq}(U)\colon (\Delta^\opp)^{|\mathcal{F}(U)|}\times \Delta\to\Sets$.
}
\end{remark}
The 2-operadic composition in $\mathbf{seq}$ 
\begin{equation}\label{compseq}
\bigotimes_{\mu\in\mathcal{F(V)}}\mathbf{seq}(P^{-1}(\mu)(\{I_\nu\}_{\nu\in \mathcal{F}(P^{-1}(\mu))}; J_\mu)\otimes \mathbf{seq}(V)(\{J_\mu\}_{\mu\in\mathcal{F}(V)}; K)\to\mathbf{seq}(\{I_\nu\}_{\nu\in \mathcal{F}(U)};K)
\end{equation}
is defined as follows. Take a total order $J_\tot$ of $\sqcup_{\mu\in \mathcal{F}(V)} J_\mu$ and a map of ordinals $W\colon J_\tot\to K$, and, for each $\nu\in \mathcal{F}(P^{-1}(\mu))$, a total order $I_{\mu,\tot}$ on $\sqcup_{\nu\in P^{-1}(\mu)} I_\nu$ and a map $W_\mu\colon I_{\mu,\tot}\to J_\mu$, we use the maps $\{W_\mu\}$ to pull-back the total order $J_\tot$ to a total order $I_\tot$ on $\sqcup_{\mu\in \mathcal{F}(V)} I_{\mu,\tot}=\sqcup_{\nu\in \mathcal{F}(U)}I_\nu$. Finally, we have the composition $W_\tot\colon I_\tot\to J_\tot\to K$, and define the composition \eqref{compseq} as $(I_\tot,W_\tot)$.
(One checks directly that the total order $I_\tot$ satisfies conditions (1)-(3) above).

\section{\sc Whiskering map for $\mathbf{seq}$}\label{tamopapp2}
The material of this Subsection was inspired by [BBM]. We construct a dg 2-operad $\mathbf{Breq}$ and a {\it whiskering map}
$$
w\colon \mathbf{Breq}\to|\mathbf{seq}|
$$
which are to the 2-operad $\mathbf{seq}$ as the operad $\mathcal{B}r_c$ and the whiskering map $w\colon\mathcal{B}r_c\to |\mathcal{L}_{c}|$
to the $\Delta$-colored lattice path operad $\mathcal{L}_c$ in sets, see [BBM, Sect. 3.1], [BB] ($c\ge 0$ is {\it the level}). Here $|\mathbf{seq}|$ denotes the condensation of the Tamarkin 2-operad $\mathbf{seq}$ in $C^\udot(\k)$ (which is a 2-operad in $C^\udot(\k)$). 

For our purposes, the condensation of $\mathbf{seq}$ in $C^\udot(\k)$ is defined, via its components, as follows:
$$
|\mathbf{seq}|(U)=\overline{\Tot}_J\big({\underline{\Tot}}_{\{I_\nu\}} \mathbf{seq}(\{I_\nu\}; J)\big)
$$
where $\underline{\Tot}$ denotes the realisation in $C^\udot(\k)$ by the simplicial variables, and $\overline{\Tot}$ denotes the normalisation in $C^\udot(\k)$ by the cosimplicial variable. Here the realisation and the totalisation are {\it non-normalised}. 

Define, for a 2-disk $U$, the degree $n$ component of the complex $\mathbf{Breq}(U)\in C^\udot(\k)$ as
\begin{equation}
\mathbf{Breq}(U)^n=\Big[\underline{\Tot}_{\{I_\nu\}}\Big(\mathbf{seq}(U)(\{I_\nu\}_{\nu\in\mathcal{F}(U)}; [0])\Big)\Big]^n
\end{equation}
These dg vector spaces lack the direct composition operation; it is fixed by the {\it whiskering map} $w$.

Let $\alpha\in \mathbf{Breq}(U)$ be an element coming from an element in $\mathbf{seq}(U)(\{I_\nu\}_{\nu\in\mathcal{F}(U)}; [0])$. By abuse of notations, we denote this element also by $\alpha$. For a 1-ordinal $J$, define $w_J(\alpha)\in \underline{\Tot}_{\{I_\nu^\prime\}}\mathbf{seq}(U)(\{I_\nu^\prime\}_{\nu\in\mathcal{F}(U)}; J)$ as follows.

Define a set $S_{\alpha, J}$ by the following data:
an element $s\in S_{\alpha, J}$ is a collection of 1-ordinals $\{I_\nu^\prime\}_{\nu\in \mathcal{F}(U)}$ and {\it surjections} of 1-ordinals $p_\nu\colon I_\nu^\prime \to I_\nu$
such that \footnote{We adopt the notation $|[n]|=n+1$, and $|J|$ is the number of elements in the ordinal $J$. More generally, for a finite set $X$ we denote by $|X|$ the number of elements in $X$.}  $$\Upsilon=\sum_{\nu,\  i\in I_\nu}(|p_\nu^{-1}(i)|-1)=|J|-1$$That is, we glue in several {\it intervals} in each $I_\nu$ such that the total number of inserted intervals, over all $\nu\in \mathcal{F}(U)$, is equal to $|J|-1$, the number of elementary intervals in the ordinal $J$. clearly $p^{-1}_\nu(i)$ is connected subordinal in $I_\nu^\prime$ which contains $|p_\nu^{-1}(i)|-1$ elementary intervals (thus, only $i\in I_\nu$ with $|p_\nu^{-1}(i)|\ge 2$ correspond to (non-empty) intervals. 

\comment
such that $\sum_{\nu\in\mathcal{F}(U)}(|I_\nu^\prime|-|I_\nu|)=|J|-1$, with a choice of an element in each $p_\nu^{-1}(i)$, $\nu\in \mathcal{F}(U), i\in I_\nu$. These chosen elements are called {\it distinguished}, the other elements are called {\it new} or {\it black}. 
\endcomment

Note that the total order $\alpha$ on $\sqcup_\nu I_\nu$ automatically pulls back to a total order $(\{p_\nu\})^*\alpha$ on the disjoint union $\sqcup_\nu I_\nu^\prime$. This total order satisfies conditions (1)-(3) of Section \ref{tamopapp1}. (Note that $(\{p_\nu\})^*\alpha$ comes with the canonical map $t\colon (\{p_\nu\})^*\alpha\to \alpha$ of ordinals). 

For each element $s\in S$, define a map of 1-ordinals 
$$
W^\prime\colon (\{p_\nu\})^*\alpha\to J
$$
as follows.

Call an element $i\in I_\nu$ essential if $\sharp p_\nu^{-1}(i)\ge 2$, clearly $p_\nu^{-1}(i)$ is a connected interval. 
Let $A_\nu$ be a generally non-connected sub-ordinal in $I_\nu^\prime$ equal to the union of pre-images $p_\nu^{-1}(i)$ over all essential $i\in I_\nu$. The images of $p^{-1}_\nu(i)$ for essential $i$ clearly remain connected subordinals in the total order $(\{p_\nu\})^*\alpha$. By an elementary interval of an ordinal $[N]$ we mean any interval $[i,i+1]$. We say that an elementary interval in $(\{p_\nu\})^*\alpha$ is essential if it belongs to the image of some $A_\nu$, and inessential otherwise.

Define $W^\prime(0)=0$, and 
\begin{equation}
\begin{cases}
W^\prime(i+1)=W^\prime(i)&\text{  if  }[i,i+1]\subset (\{p_\nu\})^*\alpha\text{  is inessential}\\
W^\prime(i+1)=W^\prime(i)+1&\text{ if  }[i,i+1]\subset (\{p_\nu\})^*\alpha\text{  is essential}
\end{cases}
\end{equation}
That is, each inessential elementary interval is mapped to a point in $J$, and each essential elementary interval is mapped isomorphically to an elementary interval in $J$. The total number of essential intervals is equal to $|J|-1$, as well as the number of newly added points (which is, by definition, $\sum_\nu\sum_{i\in I_\nu}(\sharp p_\nu^{-1}(i)-1)$). Clearly the map $W^\prime$ is a surjection. 

\comment
$$
W^\prime(a)=\ell+1
$$
if there are $\ell$ ``new'' points less than $a$. 
\endcomment

Thus, to each $\alpha\in\mathbf{Breq}(U)$ and $s\in S_{\alpha, J}$ is associated an element in $\mathbf{seq}(U)(\{I_\nu^\prime\}; J)$. This element is denoted by $w_{J,s}(\alpha)$.

\comment

\begin{itemize}
\item[(1)] 1-ordinals $\{I_\nu^\prime\}_{\nu\in \mathcal{F}(U)}$ and surjections of 1-ordinals $p_\nu\colon I_\nu^\prime \to I_\nu$
such that $\sum_{\nu\in\mathcal{F}(U)}(|I_\nu^\prime|-|I_\nu|)=|J|$,
\item[(2)] a total order on the disjoint union $\sqcup_\nu I_\nu^\prime$ satisfying (1)-(3) of Section \ref{tamopapp1}, and a surjection of 1-ordinals $P\colon  \sqcup_\nu I_\nu^\prime \to \sqcup_\nu I_\nu$ (where the total order on the r.h.s. is given by the element $\alpha$), and a map of 1-ordinals $W\colon \sqcup_\nu I_\nu^\prime\to J$,
\item[(3)] the diagram 
$$
\xymatrix{
I_\nu\ar[r]\ar[r]&\sqcup_\nu I_\nu\ar[r]&[0]\\
I_\nu^\prime\ar[r]\ar[u]^{p_\nu}&\sqcup_\nu I_\nu^\prime\ar[r]^{W}\ar[u]_{P}&J\ar[u]
}
$$
commutes, and the subdiagram
$$
\xymatrix{
\sqcup_\nu I_\nu\ar[r]&[0]\\
\sqcup_\nu I_\nu^\prime\ar[u]\ar[r]&J\ar[u]
}
$$
is a push-out. 
\end{itemize}
For $s\in S$, denote by $w_{J,s}\in\mathbf{seq}(U)(\{I_\nu^\prime\}_{\nu\in \mathcal{F}(U)}; J)$ the corresponding element.

\endcomment

Define
\begin{equation}\label{whsigns}
w_J(\alpha)=\sum_{s\in S_{\alpha,J}}(-1)^{N_s} w_{J,s}(\alpha)
\end{equation}
where the signs $N_s$ are defined just below, and define a map 
\begin{equation}\label{eqw}
\begin{gathered}
w\colon \mathbf{Breq}\to |\mathbf{seq}|\\
\alpha\to \prod_{J\in \Delta}w_J(\alpha)
\end{gathered}
\end{equation}
The map $w$ is called the {\it whiskering map}.  Note that, $w$ is a degree 0 map  of the underlying graded vector spaces.

Define the sign in front of $w_{J,s}(\alpha)$. Let $s=\{p_a\colon I_a^\prime\to I_a\}_{a\in |T|}$. For $a<_0b$ denote by $r_{a,b}$ the maximal element of $I_b^\prime$ such that $[0,r_{a,b}]<I_a^\prime$, in the sense of the total order in $(\{p_a\})^*\alpha$. Denote
\begin{equation}\label{signss1}
n_{a,b}=\sum_{i\in I_a} r_{a,b}\cdot (|p_a^{-1}(i)|-1)+\sum_{j\in [0, p_b(r_{a,b})]}(|p_b^{-1}(j)|-1)\cdot (|I_a|-1)
\end{equation}
and, further, 
\begin{equation}\label{signss2}
N_s=\sum_{a<_0 b}n_{a,b}
\end{equation}
We define the sign in \eqref{whsigns} in from of $w_{J,s}$ as $(-1)^{N_s}$.

\begin{remark}{\rm
The map $w$ is a {\it 2-operadic} counterpart of the corresponding whiskering map of [BBM, (4)], where the case of the {\it symmetric} lattice path operad is considered. On the other hand, the Tamarkin 2-operad uses, in a sense, the Joyal dual description to the lattice path operad. This is why adding of points in loc.cit. becomes adding of elementary intervals here. 
}
\end{remark}

\begin{prop}
With the signs defined as above the map $w\colon \mathbf{Breq}(T)\to |\mathbf{seq}|(T)$ is a map of complexes, for any arity 2-ordinal $T$. 
\end{prop}
\begin{proof}
It is a direct check. 
\end{proof}

The map $w$ allows to define a 2-operad structure on the 2-sequence $\{\mathbf{Breq}(U)\}_U$, inherited from the dg 2-operad structure 
on $|\mathbf{seq}|$. Explicitly, it is defined as follows. Let $P\colon U\to V$ be a map of 2-disks. For a minimal ball $\mu\in \mathcal{V}$, denote by $P^{-1}(\mu)$ its pre-image, which is 2-subdisk of $U$. We have to define
$$
M_P\colon \bigotimes_{\mu\in \mathcal{F}(V)}\mathbf{Breq}(P^{-1}(\mu))(\{I_\nu\}_{\nu\in\mathcal{F}(P^{-1}(\mu))})\otimes \mathbf{Breq}(\{J_\mu\}_{\mu\in\mathcal{F}(V)})\to\mathbf{Breq}(U)(\{I_\nu\}_{\nu\in \mathcal{F}(U)})
$$
Let $\alpha_\mu\in \mathbf{seq}(P^{-1}(\mu))(\{I_\nu\}_{\nu\in\mathcal{F}(P^{-1}(\mu))};[0])$, $\beta\in\mathbf{seq}(V)(\{J_\mu\}_{\mu\in \mathcal{F}(V)};[0])$. Define
\begin{equation}\label{opbrec}
M_P((\otimes_{\mu\in\mathcal{F}(V)}\alpha_\mu)\otimes\beta)=m_P(\otimes_{\mu\in\mathcal{F}(V)}(w_{J_\mu}\alpha_\mu)\otimes \beta)
\end{equation}
in $|\mathbf{seq}|$, and extend it by linearity. (Here $m_P$ is the operadic composition in $|\mathbf{seq}|$ associated with $P\colon U\to V$). Essentially, the in the rule \eqref{opbrec} we replace each $\alpha_\mu$ by its whiskering image $w_{J_\mu}\alpha_\mu$, after which the operadic composition in $|\mathbf{seq}|$ is defined. 

\begin{prop}\label{propv21}
The map $w\colon \mathbf{Breq}\to|\mathbf{seq}|$, see \eqref{eqw}, is a map of dg operads.
\end{prop}
\begin{proof}
Let $P\colon U\to V$ be as above. 
One has to prove that
$$
w(M_P((\otimes_{\mu\in \mathcal{F}(V)}\alpha_\mu)\otimes \beta))=m_P((\otimes_{\mu\in \mathcal{F}(V)}w\alpha_\mu)\otimes w\beta)
$$
or, in the notations as above,
\begin{equation}
w(m_P(\otimes_{\mu\in\mathcal{F}(V)}(w_{J_\mu}\alpha_\mu)\otimes \beta))=m_P(\otimes_{\mu\in\mathcal{F}(V)}(w\alpha_\mu)\otimes w\beta)
\end{equation}
Note that $m_P(\otimes_{\mu\in\mathcal{F}(V)}(w_{J_\mu}\alpha_\mu)\otimes \beta)=m_P(\otimes_{\mu\in\mathcal{F}(V)}(w\alpha_\mu)\otimes \beta)$. The statement would follow from the identity
\begin{equation}\label{hidden}
m_P(\otimes_{\mu\in\mathcal{F}(V)}(w\alpha_\mu)\otimes w_L\beta)=w_L(m_P(\otimes_{\mu\in\mathcal{F}(V)}(w\alpha_\mu)\otimes \beta))
\end{equation}
for any ordinal $L$.

We firstly prove \eqref{hidden} up to signs, and then compare the signs. 

The rhs of \eqref{hidden} is a sum over the set $S_1$ of surjective maps of ordinals
$$
I^\pprime_{\mu\nu}\xrightarrow{p^\prime_{\mu\nu}}I_{\mu\nu}^\prime\xrightarrow{p_{\mu\nu}}I_{\mu\nu}
$$
such that for a given $\mu\in\mathcal{F}(V)$, the preimages of all maps $p_{\mu\nu}\colon I_{\mu\nu}^\prime\to I_{\mu\nu}$ produce $|J_\mu|-1$ new elementary intervals (that is, they are in 1-to-1 correspondence with the elementary intervals of the ordinal $J_\mu$), and the preimages of $p^\prime_{\mu\nu}$, for all $\mu$ and $\nu$, produce $|L|-1$ new elementary intervals. 

The lhs of \eqref{hidden} is a sum over the set $S_2$ of surjective maps of ordinals
$$
J_\mu^\prime\xrightarrow{q_\mu}J_\mu\text{\ \ and  }p_{\mu\nu}^\pprime\colon I_{\mu\nu}^\pprime\to I_{\mu\nu}
$$
such that the preimages of $q_\mu$ give $|L|-1$ new elementary intervals (over all $\mu$), and for a given $\mu$ the preimages of the maps $p^\pprime_{\mu\nu}$ produce $|J_\mu^\prime|-1$ new elementary intervals.

We provide maps $S_1\to S_2$ and $S_2\to S_1$ which are inverse to each other. 

To get an element of $S_2$ out of an element of $S_1$, we have to firstly define a surjective map of ordinals $J_\mu^\prime\to J_\mu$.
To get $J_{\mu\nu}^\prime$ we consider the preimages with respect to the map $p^\prime_{\mu\nu}$ of all $|J_\mu|-1$ elementary intervals (``new'' with respect to $p_{\mu\nu}$)  in $I_{\mu\nu}^\prime$ {\it and} all other preimages of elements in $I_{\mu\nu}^\prime$ which contain more that 1 preimages. Here we mean that $|L|-1$ new intervals obtained as preimages over all $p^\prime_{\mu\nu}$ are ``distributed'' by $\mu$ (when for given $\mu$ all possible $\nu$ are considered). The totality of such new intervals in $I^\pprime_{\mu\nu}$ constitutes an ordinal $J_\mu^\prime$, the surjective map $p_{\mu\nu}$ is obtained when the new intervals from $L$ are contracted.  
The map $p_{\mu\nu}^\pprime$ in $S_2$ is obtained as the composition $p_{\mu\nu}^\pprime=p_{\mu\nu}\circ  p_{\mu\nu}^\prime$.

To get an element of $S_1$ out of an element of $S_2$ we have to decompose $p^\pprime_{\mu\nu}$ as $p_{\mu\nu}\circ p_{\mu\nu}^\prime$. The new intervals in the preimage of $p_{\mu\nu}^\pprime$ are some of the intervals in $J_\mu^\prime$ (call them $J_{\mu\nu}^{prime}$). One has a canonical embedding $J^\prime_{\mu\nu}\subset J^\prime_\mu$, and define $J_{\mu\nu}$ as the contraction of those elementary intervals of $J_{\mu\nu}^\prime$ whose image under (the restriction of) $q_\mu$ is a point. Define $I_{\mu\nu}^\prime$ as the result of contraction of $J_{\mu\nu}^{\prime}$ in $I^\pprime_{\mu\nu}$. By construction, it gives a factorisation of $p_{\mu\nu}^\pprime$ as this projection $I_{\mu\nu}^\pprime\to I_{\mu\nu}^\prime$ followed by a surjection $I_{\mu\nu}^\prime\to I_{\mu\nu}$. 

One checks directly that these two assignments are inverse to each other.

It remains to check the signs, which is straightforward, by use of \eqref{signss1}, \eqref{signss2}.

\end{proof}

Note that the signs in \eqref{whsigns} are uniquely determined from the requirement that $w\colon \mathbf{Breq}\to |\mathbf{seq}|$ is a map of dg operads. 

\begin{prop}\label{whiskprop}
For any 2-disk $U$, the map of complexes $w\colon \mathbf{Breq}(U)\to |\mathbf{seq}|(U)$ is a quasi-isomorphism of complexes. Consequently, the map $w\colon \mathbf{Breq}\to |\mathbf{seq}|$ is a quasi-isomorphism of dg operads. 
\end{prop}
\begin{proof} (Compare with [BBM, Th. 3.10])
Fix a 2-disk $U$.
We consider $|\mathbf{seq}|(U)$ as a bicomplex $L^{\udot\udot}$, for which the $\ell$-th column is equal to $\underline{\Tot}(\mathbf{seq}(U)(\{I_\nu\};[\ell]))$ ($\ell$ is fixed). (Thus, this bicomplex seats in the quarter $\{(x,y)| x\ge 0, y\le 0\}$, where $x=\ell$). Denote by $\underline{\Tot}(\mathbf{seq}(U)(\{I_\nu\};[\ell]_s))$, $0\le s\le \ell$ the subcomplex in $\underline{\Tot}(\mathbf{seq}(U)(\{I_\nu\};[\ell]))$ spanned by maps $\sqcup_\nu I_\nu \to [\ell]$ which factor as $\sqcup_\nu I_\nu\to [0]\xrightarrow{i_s} [\ell]$, where $i_s(\{0\})=\{s\}$.

\begin{prop}\label{lemmaregf}
For any $\ell\ge 0$, the cohomology groups of the complex $\underline{\Tot}(\mathbf{seq}(U)(\{I_\nu\};[\ell]))$ are non-zero only in degree 0, and $H^0$ is isomorphic to the coinvariants of the cyclic group $\mathbb{Z}_{\ell+1}$ 
$$
\Big(\bigoplus_{0\le s\le \ell}\ H^0(\underline{\Tot}(\mathbf{seq}(U)(\{I_\nu\};[\ell]_s)))\Big)_{\mathbb{Z}_{\ell+1}}\simeq H^0(\underline{\Tot}(\mathbf{seq}(U)(\{I_\nu\};[0])))
$$
\end{prop}
As follows from Theorem \ref{theoremb3}, $\dim H^0(\underline{\Tot}(\mathbf{seq}(U)(\{I_\nu\};[0])))=1$. 

We prove Proposition \ref{lemmaregf} in Section \ref{sectionlemmaregf} below.

\hspace{2mm} 

We turn back to the proof of Proposition \ref{whiskprop}. 
The idea is to use the {\it complete} convergence theorem for spectral sequences (applied to the column filtration), see [W, Th. 5.5.10, Th. 5.5.11]. 
Consider its column filtration $\Phi_\ell=\prod_{i\ge \ell}\oplus_{j\le 0}L^{ij}$, it is an ascending, exhaustive, and {\it complete} filtration. Moreover, it is {\it regular} by Proposition \ref{lemmaregf} and Theorem \ref{theoremb3}  (see [W, 5.2.10]). Then this spectral sequence converges to the cohomology of the bicomplex, by [W, Theorem 5.5.10]. \footnote{Note that for the case of {\it non-completed} complex (the direct sum totalisation) the column filtration gives rise to divergent spectral sequence. Contrary, the row filtration spectral sequence converges by [W, Theorem 5.5.1]. In the {\it completed} case, [W, Theorem 5.5.1] is not applicable for the row filtration, because this filtration is not  {\it exhaustive}. Thus, for the bicomplex of Proposition \ref{whiskprop}, only the column filtration gives rise to convergent spectral sequence. }

Denote the column filtration on $|\mathbf{seq}|(U)$ by $\Phi$, it is a descending filtration $\Phi_0\supset \Phi_1\supset \Phi_2\supset\dots$. Then $E_0^{\ell *}=\underline{\Tot}(\mathbf{seq}(U)(\{I_\nu\};[\ell]))$,
the differential $d_0$ is the differential therein. The term $E_1^{pq}$ is given by Proposition \ref{lemmaregf}.

Consider the projection ${proj}: |\mathbf{seq}|(U)\to \mathbf{Breq}(U)$ mapping $0$-th column by the identity map, and the columns with $\ell>0$ to 0. It is a map of complexes (but certainly not a map of operads). One has ${proj}\circ w=\id$, so it is enough to prove that ${proj}$ is a quasi-isomorphism. 

We show that $proj$ is an isomorphism on the term $E_2^{**}$. 
By [W, Theorem 5.5.11], it implies that $proj$ induces an isomorphism on cohomology. 

The differential $d_1$ is the cochain differential which is the alternated sum of the face maps $[\ell]\to [\ell+1]$.
Thus, the complex $(E_1^{*q}, d_1)$ is the cochain complex of the {\it constant} (by Proposition \ref{lemmaregf}) cosimplicial vector space $H^q(\underline{\Tot}(\mathbf{seq}(U)(\{I_\nu\};[0])))$. (Hence, $E_1^{*q}=0$ for $q\ne 0$, by Theorem \ref{theoremb3}, but it is not important for this argument). 

Thus, the complex $(E_1^{*q},d_1)$ is 
$$
0\to X\xrightarrow{d_0-d_1}X\xrightarrow{d_0-d_1+d_2}X\to\dots
$$
where each $d_i=\id_X$ (here $X=H^q(\underline{\Tot}(\mathbf{seq}(U)(\{I_\nu\};[0])))$, and, in fact, $X=0$ for $q\ne 0$). This complex looks is
$$
0\to X\xrightarrow{0}X\xrightarrow{\id}X\xrightarrow{0}X\xrightarrow{\id}\dots
$$
and its cohomology is equal to $X$ in degree 0 and vanishes otherwise: $$E_2^{pq}=\begin{cases}0&p\ne0\\
H^q(\underline{\Tot}(\mathbf{seq}(U)(\{I_\nu\};[0])))&p=0
\end{cases}
$$
(Due to Theorem \ref{theoremb3}, $E_2^{pq}=0$ unless $p=0,q=0$, and $E_2^{00}=\k$). 

It follows that, at the term $E_2$, the cohohomology vanish except for 0th column, and, therefore, $proj$ induces an isomorphism at the terms $E_2$. Now Proposition \ref{whiskprop} follows from [W, 5.5.11].

\end{proof}

\section{\sc Normalised $\mathbf{Breq}$ and a new proof of contractibility of Tamarkin's 2-operad}\label{sectionb3}
\subsection{\sc }
Denote by $\Nor(-)$ the normalised realisation of a (poly)simplicial set in $C^\udot(\k)$. In particular, one gets complexes
$$
\Nor(\mathbf{Breq})(U):=\Nor_{\{I_\nu\}}(\mathbf{seq}(U)(\{I_\nu\}_{\nu\in \mathcal{F}(U)}; [0])
$$
These dg vector spaces form a dg 2-operad, denoted by $\Nor(\mathbf{Breq})$. One gets a zig-zag of two operads
$$
|\mathbf{seq}|\overset{w}{\leftarrow} \mathbf{Breq}\rightarrow  \Nor(\mathbf{Breq})
$$
The left-hand side whiskering map is a quasi-isomorphism by Proposition \ref{whiskprop}, and the right-hand side map is a quasi-isomorphism by basic properties of the normalisation. 

As a conclusion, the contractibility of the 2-operad $\Nor(\mathbf{Breq})$ implies the contractibility of the dg condensation of Tamarkin's 2-operad. 

Note that, from the point of view we adapt here, the main advantage of passing from $|\mathbf{seq}|$ to $\Nor(\mathbf{Breq})$ is computational: this passage greatly simplifies, as we are going to show, the computation of the cohomology of the operad.

\begin{theorem}\label{theoremb3}
Let $U$ be a 2-ordinal. The cohomology of $\Nor(\mathbf{Breq})(U)$ is isomorphic to $\k$ in degree 0 and vanishes for higher (negative) degrees. Moreover, the 2-operad $|\mathbf{seq}|$ is quasi-isomorphic to $\k$.
\end{theorem}

We prove Theorem \ref{theoremb3} below in Section \ref{sectionb3}.

\vspace{2mm}

Note that the complexes $\Nor(\mathbf{Breq})(U)$ are $\mathbb{Z}_{\le 0}$-graded. The terms are linear combinations of all total orders of $\sqcup I_\nu$, satisfying (1)-(3) of Section \ref{tamopapp1}, where $\nu\in\mathcal{F}(U)$, $I_\nu\in \Delta$.
Such total order has degree $-\sum_\nu \ell_\nu$, where $I_\nu=[\ell_\nu]$. We call different $\nu$ {\it colors}. 
Moreover, we consider the {\it normalised} complex, which means the quotient complex
$$
\Nor(\mathbf{Breq})(U)=\mathbf{Breq}(U)/\mathbf{Breq}(U)_0
$$
where $\mathbf{Breq}(U)_0$ is the subcomplex formed by the linear combinations of total orders, in which {\it at least two consecutive elements have the same color}. Indeed, any such linear combination is the image of the corresponding degeneracy map. Note that the number of all elements of color $\nu$ in the total order is equal to $\ell_\nu+1=|I_\nu|$. \\

Before dealing with the general case, consider some examples.

\begin{example}{\rm
Let the 2-disk $U$ be a single length $n$ column, $U=(n)$. Then the colors $\nu_1<\nu_2<\dots<\nu_n$ form a totally ordered set. By condition (2) of Section \ref{tamopapp1}, in the total order $I_{\nu_1}<I_{\nu_2}<\dots<I_{\nu_n}$. Thus, there is the only total order which fulfils (1)-(3):
$$\underset{\ell_1+1}{\underbrace{\nu_1,\dots,\nu_1}}\underset{\ell_2+1}{\underbrace{\nu_2,\dots,\nu_2}}\dots \underset{\ell_n+1}{\underbrace{\nu_n,\dots,\nu_n}}
$$
On the other hand, this element belongs to $\mathbf{Breq}(U)_0$, if at least one $\ell_i>0$. It follows that the only non-zero element in the quotient-complex is
$$
\nu_1\nu_2\dots\nu_n
$$
for which $I_{\nu_i}=[0]$ for $i=1,\dots,n$. Its degree is 0. The cohomology is $\k$ and is concentrated in degree 0. 
}
\end{example}

\begin{example}{\rm
Consider the case $U=(1,1)$, it has two columns, each of length 1. Denote by $\nu_1$ and $\nu_2$ the colors, and assume that $\pi_U(\nu_1)<\pi_U(\nu_2)$. Due to (1)-(3) of Section \ref{tamopapp1}, any total order is of one of the following types:
$$
\begin{gathered}
\nu_1\dots\nu_1\nu_2\dots\nu_2\\
\nu_2\dots\nu_2\nu_1\dots\nu_1\\
\underset{a}{\underbrace{\nu_2\dots\nu_2}}\nu_1\dots\nu_1\underset{\ell_2+1-a}{\underbrace{\nu_2\dots\nu_2}}, 1\le a\le \ell_2
\end{gathered}
$$
The first two configurations are 0 in the normalised complex unless $\ell_1=\ell_2=0$. The non-zero cases are $\nu_1\nu_2$ and $\nu_2\nu_1$, they both have degree 0. The third configuration is non-zero only when it is $\nu_2\nu_1\nu_2$, in this case $\ell_1=0,\ell_2=1$, and the element has degree -1. The complex we get is
$$
0\to\underset{\text{degree }-1}{\k}\xrightarrow{1\to(1,-1)}\underset{\text{degree }0}{\k\oplus\k}\to 0
$$
The cohomology is $\k$ in degree 0. 
}
\end{example}

\subsection{\sc The case of a general 2-disk $U$ with two columns}\label{ssecnorbrec2}
Let $U=(n_1,n_2)$ be a disk with 2 columns. Denote by $\nu_1<\dots<\nu_{n_1}$ the colors of the left-hand column, and by $\lambda_1<\dots<\lambda_{n_2}$ the colors of the right-hand column. 

We start with the case $n_2=1$ and arbitrary $n_1$. 
All normalised configurations satisfying (1)-(3) of Section \ref{tamopapp1} can be listed as follows:
\begin{itemize}
\item[(type 1)] the color $\lambda=\lambda_1$ has multiplicity 1 in the total order, such normalised configurations are 
\begin{equation}\label{type1el}
\nu_1\nu_2\dots\nu_a\lambda\nu_{a+1}\dots\nu_{n_1}
\end{equation}
Such configurations have degree 0.
\item[(type 2)] the color $\lambda$ has multiplicity $N>1$. Such total order has the following form:
\begin{equation}\label{type2el}
X_1\lambda X_2\lambda X_3\dots X_{k}\lambda X_{k+1}
\end{equation}
where each $X_i$ is $\nu_a\nu_{a+1}\dots\nu_{a+b}$, $b\ge 0$, and 
$X_1X_2\dots X_k=\nu_1\nu_2\nu_3\dots\nu_{n_1}$. Such configuration has degree $-N+1$. 
\end{itemize} 
The differential of a configuration is a sum, with appropriate signs, of configurations  obtained by removing of an element whose color has multiplicity $>1$, when after such removal we get two elements of the same color staying in turn, such configuration is considered with 0 coefficient. 

For example,
$$
d(\nu_1\lambda\nu_2\nu_3\lambda\nu_4)=\nu_1\nu_2\nu_3\lambda\nu_4-\nu_1\lambda\nu_2\nu_3\nu_4
$$

It is clear that the differential of a configuration of type 2 may be a configuration of type 1, whilst the differential is 0 at any configuration of type 1. 

We provide now an algebraic description of the complex we get. 

Let $A=\k[t]$ be the associative algebra over $\k$ of polynomials in one variable. Consider the normalised (reduced) homological Hochschild complex 
$\Hoch_\ldot(A,A\otimes A)$ with coefficients in the rank 1 free bimodule\footnote{Certainly, this complex is isomorphic to the reduced bar-resolution of $A$ in the category of $A$-bimodules.}. Its elements are 
\begin{equation}\label{hochhom}
t^{M_1}\otimes (t^{a_1}\otimes t^{a_2}\otimes\dots\otimes t^{a_k})\otimes t^{M_2},\ \text{ all } a_i>0, M_1,M_2\ge 0
\end{equation}
Such element has degree $-k$.

Moreover the complex $\Hoch_\ldot(A,A\otimes A)$ is a direct sum of {\it complexes}
$$
\Hoch_\ldot(A,A\otimes A)=\bigoplus_{N\ge 0}\Hoch_\ldot^{(N)}(A,A\otimes A)
$$
where $\Hoch_\ldot^{(N)}(A,A\otimes A)$ is formed by the elements \eqref{hochhom} for which 
$$
M_1+a_1+a_2+\dots+a_k+M_2=N
$$
\begin{lemma}
Each complex $\Hoch_\ldot^{(N)}(A,A\otimes A)$ is quasi-isomorphic to $\k[0]$.
\end{lemma}
\begin{proof}
It is clear that the higher cohomology of $\Hoch_\ldot(A,A\otimes A)$, because $A\otimes A$ is a free bimodule, and degree 0 cohomology is $A=\k[t]$. This degree 0 cohomology has a 1-dimensional subspace of grading $N$, generated by $t^N$. 
\end{proof}

Turn back to the 2-disk $U=(n_2,1)$. One identifies the complex of all configurations of type 1 and type 2 with the complex 
$\Hoch_\ldot^{(N)}(A,A\otimes A)$ for $N=n_1$. Indeed, one associates with type 1 element \eqref{type1el} the element
$$
t^{M_1}\otimes t^{M_2}\in \Hoch_\ldot^{(N)}(A,A\otimes A), M_1=a, M_2=n_1-a
$$
and one identifies with type 2 element \eqref{type2el} the element
$$
t^{M_1}\otimes (t^{a_1}\otimes t^{a_2}\otimes\dots\otimes t^{a_k})\otimes t^{M_2}\in \Hoch_\ldot^{(N)}(A,A\otimes A)
$$
where $M_1=\deg(X_1), a_1=\deg(X_2), \dots, a_k=\deg(X_k), M_2=\deg(X_{k+1})$, where by $\deg(-)$ we mean the number of elements in the substring. 

One checks that the differentials act accordingly, and thus we have identified the complex $\Nor(\mathbf{Breq})(U)$ with the (reduced) Hochschild complex $\Hoch_\ldot^{(N)}(A,A\otimes A)$ for $N=n_1$. It completes the proof of contractibility, for $U=(n_1,1)$.

\vspace{4mm}

Now consider the general case of $U=(n_1,n_2)$. In this case, we have, for each $\lambda_i$, the corresponding complex. We identify $\Nor(\mathbf{Breq})(U)$ with the degree $N=n_1$ part in the tensor product:
\begin{equation}\label{hochiter}
\big[\underset{n_2\text{ factors}}{\Hoch_\ldot(A,A\otimes A)\otimes_A\dots\otimes_A \Hoch_\ldot(A,A\otimes A)}\big]^{(n_1)}
\end{equation}
Here we mean the following. The complex 
$$
\underset{n_2\text{ factors}}{\Hoch_\ldot(A,A\otimes A)\otimes_A\dots\otimes_A \Hoch_\ldot(A,A\otimes A)}
$$
is graded, by the total sum of degrees of $t$ in an element. Thus, for any $N\ge 0$, we get a degree $N$ subcomplex. 
\begin{lemma}
For each $N\ge 0$, the complex
$$
\big[\underset{n_2\text{ factors}}{\Hoch_\ldot(A,A\otimes A)\otimes_A\dots\otimes_A \Hoch_\ldot(A,A\otimes A)}\big]^{(N)}
$$
is quasi-isomorphic to $\k[0]$, $n_2\ge 1$. 
\end{lemma}
\begin{proof}
Each factor $\Hoch_\ldot(A,A\otimes A)$ is a complex of free $A$-bimodules. It follows that 
$$
H^\udot\Big(\Hoch_\ldot(A,A\otimes A)^{{\otimes_A} n}\Big)=\Big[H^\udot(\Hoch_\ldot(A,A\otimes A))\Big]^{\otimes_A n}
$$
Then the cohomology is $A\otimes_A A\otimes_A\dots\otimes_A A=A$ and is concentrated in degree 0. Then each graded component has 1-dimensional cohomology. 
\end{proof}
The identification of the complex \eqref{hochiter} with $\Nor(\mathbf{Breq})(U)$ is fairly analogous to the case $n_2=1$, considered above. 

It completes the proof of contractibility for $U=(n_1,n_2)$. 

\qed

\subsection{\sc The contractibility of $\Nor(\mathbf{Breq})(U)$ for general $U$}\label{ssecnorbrecgen}
Let $U=(n_1,\dots,n_k)$ be a general 2-disk. With each minimal ball $\nu$ is associated a 1-ordinal $I_\nu$, and hence a differential after the dg realisation by $I_\nu$. The complex $\Nor(\mathbf{Breq})(U)$ is the (sum) total complex of a poly-complex, with differentials associated with different minimal balls $\nu$. Denote by $d_i$ the sum of these differentials for all $\nu$ in the $k$-th column, $1\le i\le k$. They clearly commute with each other (moreover, the differentials associated with different $\nu$ commute). 

We use spectral sequences for computing the cohomology of $\Nor(\mathbf{Breq})$, as follows. 

We organise our polycomplex as a bicomplex, in which $D_1=d_k, D_2=d_1+\dots+d_{k-1}$, and compute the cohomology of $D_1$ at  first. The spectral sequence converges, as the polycomplex is $\mathbb{Z}_{\le 0}$-graded. One computes the cohomology of $D_1$ similarly with our computation in Section \ref{ssecnorbrec2}, via the Hochschild complex interpretation. The result is that the non-zero cohomology exist only  in degree 0 by the rightmost column, and all colors from the rightmost column appear with multiplicity 1 in cohomology. 

Next, we have to compute cohomology of the differential $d_1+\dots+d_{k-1}$. Set $D_1=d_{k-1}, D_2=d_1+\dots+d_{k-2}$, and use the (convergent) spectral sequence which computes the cohomology of $D_1$ at first. Note that the elements from the the $k$-th column do not affect this computation. Once again, we get that the cohomology are non-zero only when the colors of all elements of $(k-1)$-st column have multiplicity 1, and is concentrated in degree 0. This we get that the multiplicities of all elements from the two rightmost columns have multiplicity 0 in cohomology.

We repeat arguing in this way, so that at the $\ell$-th step $D_1=d_{k-\ell+1}, D_2=d_1+\dots+d_{k-\ell}$, $1\le \ell\le k$.
Finally we get that the cohomology of $\Nor(\mathbf{Breq})(U)$ are non-zero only in degree 0, and for the representing configurations the multiplicities of all colors are equal to 1. 

It remains to show that all such configurations in degree 0 are cohomologous, which we leave to the reader. We get a projection $|\mathbf{seq}|(U)\to \k$, which is a quasi-isomorphism, for any $U$. It is straightforward that it gives to a map of 2-operads $|\mathbf{seq}|\to\underline{\k}$. 

Theorem \ref{theoremb3} is proven.

\qed

\comment

At first, $d_1=0$ (for normalised configurations satisfying (1)-(3) of Section \ref{tamopapp1}).
Consider the spectral sequence of the bicomplex, whose first differential is $D_1=d_1+d_2$, and the second $D_2=d_3+d_4+\dots+d_k$, which firstly computes the cohomology of $D_1$. Our complexes are $\mathbb{Z}_{\le 0}$-graded, so the spectral sequence converges by dimensional reasons. 

One can compute the cohomology of $D_1$ similarly to the computation in Section \ref{ssecnorbrec2}. The only difference is that we have ``several connectness compontents'' by which the images of $I_\nu$ for $\nu$ in the first two columns is divided by images of $I_\nu$ with $\pi(\nu)\ge 3$. It results in the tensor product over $\k$ of several complexes for each of which the cohomology can be computed as in Section \ref{ssecnorbrec2}.

To perform the inductive procedure, we interpret this computation as follows.

In the notations as above, denote by $K_2(U)$ the normalised totalisation of $\mathbf{seq}_2(U)(\{I_\nu\};[0])$ where $\mathbf{seq}_2(U)$ denotes the set defined as $\mathbf{seq}(U)$ with condition (3c) removed, when $\pi(\nu_1)=1,\pi(\nu_2)=2$. One has the natural embedding $i_2\colon K_2(U)\to\Nor(\mathbf{Breq})(U)$.

\begin{lemma}
The embedding $i_2$ is a quasi-isomorphism of complexes.
\end{lemma}
\begin{proof}
We prove that $\Cone(i_2)$ is acyclic, by the spectral sequence of bicomplex, computing the cohomology of $D_1$ at first. The computation itself is fairly the same as above.
\end{proof}

So we replace $\Nor(\mathbf{Breq})(U)$ by $K_2(U)$. In has that advantage that in any configuration of $K_2(U)$ any color $\nu$ with $\pi(\nu)=1,2$ {\it occurs with multiplicity 1}.

Next, consider the complex $K_3(U)$, defined as the normalised totalisation of $\mathbf{seq}_3(U)(\{I_\nu\};[0])$, where $\mathbf{seq}_3(U)$ is the set of configurations defined as $\mathbf{seq}(U)$, with (3c) removed for $\pi(\nu)=1,2,3$. There is an embedding 
$i_3\colon K_3(U)\to K_2(U)$. We prove that $i_3$ is quasi-isomorphism, for which we prove that the cone $\Cone(i_3)$ is acyclic. Now we set $D_1=d_1+d_2+d_3, D_2=d_4+\dots+d_k$, and make use the (convergent) spectral sequence in which the cohomology of $D_1$ at first. The cohomology of $D_1$ is computed similarly to the computation for $K_2(U)$, and so on.

Finally we define $K_k(U)$ in which (3c) is removed for any $1\le \nu_1,\nu_2\le k$, and prove similarly that the natural embedding $i_k\colon K_k(U)\to K_{k-1}(U)$ is a quasi-isomorphism.

We get a chain of quasi-isomorphisms
$$
K_k(U)\overset{\sim}{\rightarrow}K_{k-1}(U)\overset{\sim}{\rightarrow}\dots \overset{\sim}{\rightarrow}K_2(U)\overset{\sim}{\rightarrow}\Nor(\mathbf{Breq}(U)
$$
Note that in $K_k(U)$ {\it each} color $\nu$ has multiplicity 1, and that $K_k(U)$ is in fact a vector space concentrated in degree 0. 

\endcomment

\section{\sc A proof of Proposition \ref{lemmaregf}}\label{sectionlemmaregf}
It is enough to compute the cohomology of the {\it normalised} chain complex $\underline{\Nor}(\mathbf{seq}(\{I_\nu\}; [\ell]))$, for a fixed $\ell$. The proof is based on ideas employed in the computation of $\underline{\Nor}(\mathbf{seq}(\{I_\nu\};[0]))$ in Section \ref{sectionb3}.

First of all, we stress the following difference between the normalised chains for $\ell=0$ and for $\ell\ge 1$. For the case $\ell\ge 1$ one may have two {\it consecutive} elements of the same color $\lambda$ in the total order on $\sqcup I_\nu$ if their images under the map of 1-ordinals $W\colon \sqcup I_\nu\to [\ell] $ are {\it different} elements of $[\ell]$. Indeed, such elements are not boundaries. 

Therefore, the situation is more complicated when $\ell\ge 1$. We proceed as follows.

For a monomial element $\omega$ of $\underline{\Nor}(\mathbf{seq}(\{I_\nu\};[\ell]))$ and for a color $\lambda$ denote $$m(\lambda)=\text{(the multiplicity of $\lambda$ in $\omega$)}-1$$  The {\it total multiplicity degree} $M(\omega)$ is defined as $\sum_\lambda m(\lambda)$. 

We compute the cohomology of $\underline{\Nor}(\mathbf{seq}(\{I_\nu\};[\ell]))$ using spectral sequence, associated with the descending filtration $F_0\subset F_{-1}\subset F_{-2}\subset \dots $, where $F_{-k}$ is formed by linear combinations of monomials $\omega$ with $M(\omega)\le k$. The components $E_0^{pq}\ne 0$ for $p\le 0, q\le 0$. 
It converges to the cohomology of $\underline{\Nor}(\mathbf{seq}(\{I_\nu\};[\ell]))$ by the ``classical convergence theorem'' [W, Theorem 5.5.1]. 

The differential $d_0$ computes the cohomology in the quotient-complexes $F_i/F_{i-1}$. This amounts to computing the cohomology of the differential component acting on the subcomplexes $W^{-1}(j), j\in [\ell]$. Moreover, the total differential $d$ on $\underline{\Nor}(\mathbf{seq}(\{I_\nu\};[\ell]))$ equals to a sum of commuting differentials $d_{(j)}$, where $d_{(j)}$ acts only on $W^{-1}(j)$. It can be computed similarly to the computation in Section \ref{sectionb3}. The result is that $H^0(W^{-1}(j))=\k$, and the higher cohomology vanishes. In other words, $E_1^{p,q}=0$ for $q\ne 0$.

In particular, each color appears in $H^0(W^{-1}(j))$ with multiplicity $\le 1$. However, the same color $\lambda$ may have the multiplicity 1 for several different $j$. 

The component $d_1$ decreases this total multiplicity by 1. Any component of the differential in $\underline{\Nor}(\mathbf{seq}(\{I_\nu\};[\ell]))$ may either preserve $M(\omega)$ or decrease it by 1. Hence, the spectral sequence collapses at $E_2$.
It remains to compute the cohomology of the differential $d_1\colon E_1^{p,q}\to E_1^{p+1,q}$.

We proceed as follows: we show that the cohomology vanish unless the degree is 0, and then we compute the degree 0 cohomology directly. 

Let us have a closer look at the complex 
$$
\dots\to E_1^{-2,0}\xrightarrow{d_1} E_1^{-1,0}\xrightarrow{d_1}E_1^{0,0}\to 0
$$
Its rightmost term $E_1^{0,0}$ can be described as $\k^{N}$ where $N$ is the number of elements in  $\Delta ([\sharp \mathcal{F}(U)-1], [\ell])$. The higher terms can be described similarly, but the multiplicities of some colors (recall that a color is just an element in $\mathcal{F}(U)$) may be greater than 1. Terms in $E_1^{-p,0}$ are spanned by monomials $\omega$ with $M(\omega)=p$.

We claim that this complex is acyclic in higher degrees. It is proven by constructing a contracting homotopy $h\colon E_1^{-*,0}\to E_1^{-*-1,0}$. It is 
$$
h=\sum_{\nu\in\mathcal{F}(U)}\sum_{j\in [\ell]}\pm  i_{\nu,j}$$
where $i_{\nu,j}$ ``adds'' the color $\nu$ to $W^{-1}(j)$ if the multiplicity of $\nu$ in $W^{-1}(j)$ is 0, otherwise $i_{\nu,j}=0$.
One easily computes that for $p\ne 0$ one has $[d_1,h]=\pm p\cdot\id$, which gives the claim. 

It remains ro compute the cokernel of the map $d_1\colon E_1^{-1,0}\to E_1^{0,0}$. 
For any monomial $\omega$ in $E_1^{-1,0}$ the multiplicities of all colors except for a single color $\lambda$ are equal to 1, whence the multiplicity of $\lambda$ is 2. Assume that the color $\lambda$ appears with multiplicities 1 in $W^{-1}(i)$ and $W^{-1}(j)$, $0\le i<j\le \ell$. Denote by $\omega(\lambda,i)$ the monomial obtained from $\omega$ by removing the color $\lambda$ in $W^{-1}(i)$, define similarly $\omega(\lambda, j)$. Then 
$$
d_1(\omega)=\omega(\lambda,i)-\omega(\lambda,j)
$$
Thus, in the cokernel one can move {\it any color} from {\it any} $W^{-1}(i)$ (where it has multiplicity 1) to {\it any} other $W^{-1}(j)$.

Consequently, we can move all colors to $W^{-1}(j)$, for some fixed $j$, so that $W^{-1}(i)=\varnothing$ for $i\ne j$, and get the same element in $H^0$. Here $j$ can be arbitrary, this is we get the coinvariant description by the cyclic group $\mathbb{Z}_{\ell+1}$, replacing $j$ by any other $j^\prime$.

\qed

\section{\sc Reminder on the twisted tensor product 2-operad}
For small dg categories over a field $\k$, the {\it twisted tensor product} $C\sotimes D$ [Sh1] is defined as the left adjoint to the functor $D\mapsto \Coh_\dg(D,E)$ (for fixed small dg category $E$ over $\k$):
$$
\Hom(C\sotimes D,E)\simeq \Hom(C,\Coh_\dg(D,E))
$$
Here the dg category $\Coh_\dg(D,E)$ is defined as follows: its objects are dg functors $F\colon D\to E$, and the morphisms $F\to G$ are defined as the coherent natural transformations $F\Rightarrow G$. The latter is given, by definition, by the Hochschild cochains of $C$ with coefficients in the $C$-bimodule $D(F(-),G(=))$. 

The dg category $C\otimes D$ is explicitly defined in [Sh1], it has $\Ob(C)\times \Ob(D)$ as its objects, and the morphisms are {\it generated} by $\Mor(C)\otimes \id_Y$, $\id_X\otimes \Mor(D)$, $X\in\Ob(C), Y\in\Ob(D)$, and the new morphisms
$$
\varepsilon(f; g_1,\dots,g_n)\in (C\sotimes D)(X_0\times Y_0, X_1,Y_n)
$$
of degree $\deg f+\sum_{i=1}^n\deg g_i-n$, where $f\in C(X_0,X_1)$, and $g_i\in D(Y_{i-1},Y_i)$ form a composable chain of morphisms. 

The differential of $\varepsilon(f; g_1,\dots,g_n)$ is given by a Hochschild-like formula, which for $n=1$ gives
$$
[d,\varepsilon(f;g)]=(\id_{X_1}\otimes g)\circ (f\otimes \id_{Y_0})-(-1)^{|f||g|}(f\otimes \id_{Y_1})\circ (\id_{X_0}\otimes g)
$$
We also impose some relations on these generator morphisms, among which the most non-trivial is
$$
\varepsilon(f_2\circ f_1; g_1,\dots,g_n)=\sum_{i=0}^n\pm \varepsilon(f_2; g_{i+1},\dots,g_n)\circ \varepsilon(f_1; g_1,\dots,g_i)
$$
The reader is referred to [Sh1] for the complete list of relations.

It is proven [Sh2, Section 3] that the twisted tensor product makes the category of small dg categories over $\k$ a {\it skew-monoidal category} [S], which basically means that there is a one-sided associativity 
\begin{equation}\label{eqskew}
(C\sotimes D)\sotimes E\to C\sotimes (D\sotimes E)
\end{equation}
as well as one-sided unit maps
$$
\lambda_C\colon I\sotimes C\to C{\ \ \ \text{and}\ \ \ }\rho_C\colon C\to C\otimes I
$$
(here $I$ is the unit category, whose objects consist of a single element $*$, and $I(*,*)=\k$).
We emphasize that \eqref{eqskew} is not an equivalence nor  a weak equivalence (the unit maps are equivalences in our situation, but they are not assumed to be so in the formalism of skew-monoidal categories). These maps are subject to a list of axioms [S], [Sh2, Section 3]. 

We also proved [Sh1,Th.2.4]:
\begin{prop}\label{prophomotopy}
Let $C,D$ be small dg categories, cofibrant for the Tabuada closed model structure. Then the natural dg functor $p\colon C\sotimes D\to C\otimes D$, sending all $\varepsilon(f; g_1,\dots,g_n)$ to 0, is a quasi-equivalence. 
\end{prop}

\qed

Recall that the interval dg categories $I_n$, having $n+1$ objects ans $I_n(a,b)=\k$ when $a\le b$ and $=0$ otherwise, are cofibrant. 

Set 
$$
I_{n_1,\dots, n_k}=I_{n_k}\sotimes (I_{n_{k-1}}\sotimes (\dots \sotimes (I_{n_2}\sotimes I_{n_1})\dots))
$$

The twisted product 2-operad $\mathcal{O}$ is defined by its components 
$$\mathcal{O}(D)=I_{n_k}\sotimes (I_{n_{k-1}}\sotimes (\dots \sotimes (I_{n_2}\sotimes I_{n_1})\dots))(\min,\max)$$
where $D=(n_1,\dots,n_k)$ is a 2-diagram (see Section ???), $\min=(0,\dots,0), \max=(n_k,\dots,n_1)$ are the minimum and maximum objects of 
$I_{n_1,\dots, n_k}$. 

The 2-operadic compositions on the components $\mathcal{O}(D)$ are basically given by ``plugging", see [Sh2, Section 5]. 
The 2-operadic associativity follows from an analogue of the MacLane coherence theorem for skew-monoidal categories. In general, the statement of this coherence is far more complicated than its classical counter-part [LS]. Luckily, in the particular case when the unit maps $\lambda$ and $\rho$ are {\it isomorphisms}, the coherence statement is precisely the same as for the case of monoidal categories (which means that the associator is an isomorphism), see [Sh2, Prop. 3.6]. Such skew-monoidal categories are called {\it perfect} in [Sh2]. Thus, the perfectness of the skew-monoidal category of small dg categories with the twisted tensor product is essentially employed in the proof of associativity. 

We show that the twisted tensor product 2-operad acts on the dg 2-quiver $\Cat_\dg^\coh(\k)$, [Sh2, Section 5.5]. 

The following statement is easily deduced from Proposition \ref{prophomotopy}:

\begin{prop}\label{ocontr}
There is a natural map of 2-operads $p\colon \mathcal{O}\to\underline{\k}$ which is a quasi-isomorphism of dg 2-operads. 
\end{prop}

\qed

See [Sh2, Prop. 5.7]. Here the dg 2-operad $\underline{\k}$ is defined as $\underline{\k}(U)=\k$ with the tautological compositions.

\section{\sc $\Nor(\mathbf{Breq})$ is isomorphic to the twisted tensor product operad $\mathcal{O}$}
In this Section, we identify the dg 2-operad $\Nor(\mathbf{Breq})$ with dg 2-operad $\mathcal{O}$ defined via the twisted tensor product, see Section 5. As a consequenc of this identification and Proposition \ref{ocontr}, we get another proof of the contractibility of the dg condensation $|\mathbf{seq}|$ of the $\Delta$-colored Tamarkin 2-operad $\mathbf{seq}$. 

One has

\begin{prop}
The twisted tensor product 2-operad $\mathcal{O}$ is isomorphic to the 2-operad $\Nor(\mathbf{Breq})$.
\end{prop}
\begin{proof}
We assign to each monomial $\omega$ in $\Nor(\mathbf{Breq})(U)$ an element of $\mathcal{O}(U)$, in an inductive way. 
Assume $U=(n_1,\dots,n_{k-1},n)$, $U_0=(n_1,\dots,n_{k-1})$, $C=I_{n_{k-1}}\sotimes(I_{n_{k-2}}\sotimes(\dots(I_{n_2}\sotimes I_{n_1})\dots))$. We assume that we have already constructed a map $\phi\colon \Nor(\mathbf{Breq})(V)\to\mathcal{O}(V)$, for $V$ a diagram with $\le (k-1)$ columns, and construct a map $\Phi\colon \Nor(\mathbf{Breq})(U)\to\mathcal{O}(U)=(I_n\sotimes C)(\min,\max)$. 

Denote by $\lambda_1<\dots<\lambda_n$ the colors corresponded to the intervals of $[n]$, denote by $\sharp_i$ the multiplicity of $\lambda_i$ in $\sqcup_\nu I_\nu$. The idea is to assign to a fragment of $\omega$ 
\begin{equation}\label{eqqr1}
\lambda_i X_1 \lambda_i X_2 \lambda_i\dots X_\ell \lambda_i, \ \ \sharp_i>1
\end{equation}
with $\ell=\sharp_i-1$, where $X_1,\dots, X_\ell $ contains only colors from from $U_0$, the 
factor
\begin{equation}\label{eqqr2}
\varepsilon(f_i; \id\otimes \phi(X_1),\dots,\id\otimes \phi(X_\ell))\in I_{n_{k}}\sotimes(I_{n_{k-1}}\sotimes(\dots(I_{n_2}\sotimes I_{n_1})\dots))
\end{equation}
(here $f_i$ is the generator of $I_n$ corresponded to the color $\lambda_i$).

The monomial $\omega$ is subdivided into the union of the fragments of the following 3 types: (a) fragments as in \eqref{eqqr1}, when $\sharp_i>1$; one assigns to it the element \eqref{eqqr2} of  $I_{n_{k}}\sotimes(I_{n_{k-1}}\sotimes(\dots(I_{n_2}\sotimes I_{n_1})\dots))$, (b) elements $\lambda_j$ for $\sharp_j=1$; one assigns to them the corresponding $f_j\otimes\id$, (c) elements $Y$ containing only colors from $U_0$, one assigns to them $\id\otimes\phi(Y)$. After that, we take the composition of the assigned morphisms, over fragments of $\omega$. Clearly this composition is a morphism from $\min$ to $\max$ in $I_{n_{k}}\sotimes(I_{n_{k-1}}\sotimes(\dots(I_{n_2}\sotimes I_{n_1})\dots))$, and, hence, an element in $\mathcal{O}$. 

This inductively defined map $\Nor(\mathbf{Breq})(U)\to\mathcal{O}(U)$ is an isomorphism of complexes. 
One can see directly that it agrees with the operadic composition. 
\end{proof}

\bigskip

\noindent{\small
 {\sc Euler International Mathematical Institute\\
10 Pesochnaya Embankment, St. Petersburg, 197376 Russia }}

\bigskip

\noindent{{\it e-mail}: {\tt shoikhet@pdmi.ras.ru}}

\end{document}